\documentclass[pdftex,aos]{imsart}

\RequirePackage[OT1]{fontenc}
\RequirePackage{amsthm,amsmath,amssymb,graphicx}
\RequirePackage{natbib}
\RequirePackage[colorlinks,citecolor=blue,urlcolor=blue]{hyperref}

\arxiv{math.PR/0000000}

\startlocaldefs
\numberwithin{equation}{section}
\theoremstyle{plain}

\newtheorem{theorem}{Theorem}
\newtheorem{lemma}[theorem]{Lemma}

\newtheorem{corollary}[theorem]{Corollary}
\newtheorem{proposition}[theorem]{Proposition}
\newtheorem{remark}[theorem]{Remark}
\newtheorem{definition}[theorem]{Definition}

\DeclareMathOperator*{\argmin}{argmin}
\newcommand{\E}{\mathbb{E}}

\newenvironment{enum}{
\begin{enumerate}
  \setlength{\itemsep}{1pt}
  \setlength{\parskip}{0pt}
  \setlength{\parsep}{0pt}
}{\end{enumerate}}

\let\hat\widehat
\let\tilde\widetilde
\allowdisplaybreaks

\endlocaldefs

\begin{document}

\begin{frontmatter}
\title{Efficient Nonparametric Conformal Prediction Regions
}
\runtitle{Nonparametric Conformal Prediction}

\begin{aug}
\author{\fnms{Jing} \snm{Lei}\thanksref{t1,m1}\ead[label=e1]{jinglei@andrew.cmu.edu}},
\author{\fnms{James} \snm{Robins}\thanksref{m2}\ead[label=e2]{robins@hsph.harvard.edu}}
\and
\author{\fnms{Larry} \snm{Wasserman}\thanksref{t3,m1}
\ead[label=e3]{larry@stat.cmu.edu}
}

\thankstext{t1}{Supported by NSF Grant BCS-0941518.}
\thankstext{t3}{Supported by NSF Grant DMS-0806009 and
Air Force Grant FA95500910373.}
\runauthor{J. Lei et al.}

\affiliation{Carnegie Mellon University\thanksmark{m1}
and Harvard University\thanksmark{m2}}

\address{Department of Statistics\\
Carnegie Mellon University\\
Pittsburgh, Pennsylvania 15213\\
USA\\
\printead{e1}\\
\phantom{E-mail:\ }\printead*{e3}}

\address{Department of Epidemiology\\
Harvard School of Public Health\\
677 Huntington Avenue\\
Boston, Massachusetts 02115\\
USA\\
\printead{e2}\\
}
\end{aug}

\begin{abstract}
We investigate and extend the conformal prediction
 method due to \citet{VovkGS08b}
to construct nonparametric prediction regions.
These regions
have guaranteed distribution free,
finite sample coverage, without any assumptions on
the distribution or the bandwidth.
Explicit convergence rates of the loss function are
established
for such regions under standard regularity conditions.
Approximations for
simplifying implementation and data driven bandwidth selection
methods are also discussed. The theoretical properties of our
method
are demonstrated
through simulations.
\end{abstract}



\end{frontmatter}

\section{Introduction}

\subsection{Prediction regions and density level sets}

Consider the following prediction problem:
we observe {\em iid} data
$Y_1,\ldots,Y_n\in\mathbb{R}^d$ from a distribution $P$
and we want to construct a prediction region
$C_n=C_n(Y_1,\ldots, Y_n)\subseteq \mathbb{R}^d$ such that
\begin{equation}\label{eq:valid_exp}
\mathbb{P}(Y_{n+1}\in C_n)\ge 1-\alpha
\end{equation}
for fixed $0 < \alpha < 1$
where
$\mathbb{P}= P^{n+1}
$ is the product probability measure over the $(n+1)$-tuple
$(Y_1,\dots,Y_{n+1})$.\footnote{In general,
we let $\mathbb{P}$ denote $P^n$ or $P^{n+1}$
depending on the context.}
This is equivalent to
$\E \left[P(C_n)\right] \ge 1-\alpha$
where $P(C_n)$ is the
probability mass of the random set $C_n$.
In other words,
$C_n$ traps a future independent observation $Y_{n+1}\sim P$
with probability at least $1-\alpha$.
The random set $C_n$ is called a {\em $(1-\alpha)$-prediction region}
or a {\em $(1-\alpha)$-tolerance region}.  In this paper we
will use the name ``prediction region'' for consistency
of presentation
while ``tolerance region'' is often used as a synonym in the
literature.

Prediction is a major focus of machine learning and statistics
although the emphasis is often on point prediction.
Prediction regions go beyond merely providing a point prediction and
are useful in a variety of applications
including quality control
and anomaly detection.  For example, suppose a sequence of
items is being produced or
observed. If one item falls out of the prediction region
constructed from the previous
samples, it indicates that this item is likely to be different
from the rest of the
sample and some further investigation may be necessary.

Another application of prediction regions is data description
and clustering.
Given a random sample from a distribution, it is often of
interest to ask where most of
the probability mass is concentrated.  A natural answer to this question
is the density level set $L(t)=\{y\in \mathbb{R}^d:p(y)\ge t\}$, where $p$
is the density function of $P$.  When the distribution $P$ is multimodal,
a suitably chosen $t$ will give a clustering of the underlying
distribution \citep{Hartigan75}.  When $t$ is given,
consistent estimators of $L(t)$ and rates of convergence
have been studied in detail, for example, in \citet{Polonik95,
Tsybakov97,BailloCC01,Baillo03,Cadre06,WilletN07,RigolletV09,RinaldoW10}.
It often makes sense to define $t$ implicitly using the desired
probability coverage $(1-\alpha)$:
\begin{equation}\label{eq:t_alpha}
t(\alpha)=\inf \Bigl\{ t:\ P(L(t))\ge 1-\alpha \Bigr\}.
\end{equation}
Let $\mu(\cdot)$ denote the Lebesgue measure on $\mathbb{R}^d$.
If the contour $\{y:p(y)=t(\alpha)\}$ has zero Lebesgue measure, then
it is easily shown that
$$
C^{(\alpha)}:=L(t(\alpha))=
\argmin_C \mu(C)\,,
$$
where the min is over
$\bigl\{C:\ P(C)\ge 1-\alpha\bigr\}$.
Therefore, the density based clustering problem can sometimes be formulated
as estimation of the minimum volume prediction region.

The study of prediction regions has a long history in statistics;
see, for example
\citet{Wilks41,Wald43,FraserG56,ChatterjeeP80,DiBucchianicoEM01,Cadre06,LiL08}.
For a thorough introduction to prediction regions, the reader is referred
to the books by \citet{Guttman70} and \citet{AichisonD75}.
In this paper we study a newer method due to
\citet{VovkGS08b} which we describe in Section
\ref{sec::vovk}.

\subsection{Validity and efficiency}

Let $C_n$ be a prediction region.
There are two natural criteria to
measure
its quality: {\em validity} and {\em efficiency}.  By validity we mean that
$C_n$ has the desired coverage for all $P$, whereas by efficiency we mean that
$C_n$ is close to the optimal prediction region $C^{(\alpha)}$.

\subsubsection{Validity}

By definition, a prediction region $C_n$ is a function of the
sample $(Y_1,...,Y_n)$ and hence its coverage $P(C_n)$ is a random
quantity. To formulate the notion of validity of a prediction region,
\cite{FraserG56} defined $(1-\alpha)$-prediction regions with
$\tau$-confidence for $C_n$ satisfying
\begin{equation}\label{eq:tol_conf}
\mathbb{P}(P(C_n)\ge 1-\alpha)\ge \tau.
\end{equation}
However, evaluating the exact probability in the above definition is
rarely possible.  Most work on nonparametric prediction regions validate
their methods using an asymptotic version \citep{ChatterjeeP80,LiL08}:
$$
\liminf_{n\rightarrow \infty}\mathbb{P}\left[P(C_n)\ge 1-\alpha\right]
\ge \tau.
$$

On the other hand, if a procedure $C_n$ satisfies
(\ref{eq:valid_exp}) for every distribution $P$ on $\mathbb{R}^d$
and every $n$, then
we say that $C_n$ is a
{\em distribution free prediction region}
or has {\em finite sample validity}.

\subsubsection{Efficiency}

We measure the efficiency of $C_n$ in terms of its
closeness to the optimal region $C^{(\alpha)}$.
Recall that if
$P$ has a density $p$
with respect to Lebesgue measure $\mu$,
 then the smallest region with probability content at least $1-\alpha$ is
\begin{equation}
C^{(\alpha)} = \left\{ y:\ p(y) \ge t(\alpha)\right\},
\end{equation}
where
$t(\alpha)$ is given by (\ref{eq:t_alpha}), provided
that the contour $\{y: p(y)=t(\alpha)\}$ has zero measure.
Since $p$ is unknown, $C^{(\alpha)}$ cannot be used
as an estimator but only as a benchmark in evaluating the efficiency.
We define the loss function of $C_n$
by
\begin{equation}
R(C_n) = \mu(C_n\triangle C^{(\alpha)})
\end{equation}
where $\triangle$ denotes the symmetric set difference.
Such loss functions have been used,
for example, by \citet{ChatterjeeP80} and \citet{LiL08} in
nonparametric prediction region estimation and by \citet{Tsybakov97,RigolletV09}
in density level set estimation.
Since,
$$
\mu(C_n\triangle C^{(\alpha)}) =
\mu(C_n) - \mu(C^{(\alpha)}) + 2 \mu(C^{(\alpha)}\backslash C_n) \geq
\mu(C_n) - \mu(C^{(\alpha)})\,,
$$
it follows that the symmetric difference loss
gives an upper bound on the {\em excess loss}
\begin{equation}
{\cal E} (C_n)=\mu(C_n)-\mu(C^{(\alpha)}).
\end{equation}

In \citet{ChatterjeeP80} and \cite{LiL08}, a prediction region $C_n$ is
called asymptotically minimal if
\begin{equation}
\mu(C_n\triangle C^{(\alpha)})\stackrel{P}{\rightarrow}0.
\end{equation}
However, such an asymptotic property does not specify the
rate of convergence.  While convergence rate results
are available for density level sets estimation
\citep[see][for example]{Tsybakov97,RigolletV09,MasonP09},
relatively less
is known about prediction regions until recently
\citep{Cadre06,SamworthW10}.

\subsection{This paper}

In this paper, we propose an efficient and easy to compute
prediction region with finite sample validity
and we study the rate of convergence of its loss.
To be specific,
we construct $C_n$ such that:
\begin{enum}
\item $C_n$ satisfies (\ref{eq:valid_exp}) for all
$P$ and all $n$ under {\em no} assumption other than {\em iid}.
\item  For any $\lambda>0$,
there exist constants $c_1(\lambda,p)$ and $c_2(p)$ independent of $n$, such that
\begin{equation}
  \label{eq:intro_rates}
  \mathbb{P}\left(R(C_n)\ge c_1(\lambda, p)\left(\frac{\log n}{n}\right)^
  {c_2(p)}\right)
  = O(n^{-\lambda}),
\end{equation}
for density $p$ satisfying some standard regularity conditions.
\item For any $y\in \mathbb{R}^d$, the computation cost of evaluating $\mathbf{1}(y\in C_n)$
 is linear in $n$. In other words, checking to see if a point $y$ is in the prediction region, takes linear time.
\end{enum}
The convergence rate of efficiency is described by the term
$(\log n/n)^{c_2(p)}$.  We give explicit formula of constant $c_2(p)$
in terms of the global smoothness and the local behavior of $p$ near the
contour at level $t(\alpha)$.  Its near optimality is discussed
for some important special cases.

Our prediction region is obtained by
combining the idea of conformal prediction
\citep{VovkGS08b}
with density estimation.  We first construct a conformal
prediction region that is closely related to a kernel density
estimator.  The finite sample validity is inherited from
the nature of conformal prediction regions.
Then we show that such a region, whose analytical
form may be intractable, is sandwiched by two kernel density
level sets with
carefully tuned cut-off values.  Therefore the efficiency
of the conformal prediction region can be approximated by those
of the two kernel density level sets.
As a by-product, we obtain a kernel density level set that always
contains the conformal prediction region, and hence also satisfies finite
sample validity.  This observation means that, most of the time, a kernel
density estimator will have near optimal efficiency, finite
sample validity, and even lower computational cost at the same time.
In the efficiency argument, we
refine the rates of convergence for plug-in density level sets
first developed in \cite{Cadre06}, which may be of independent
interest.

Our method involves one tuning parameter which is the bandwidth
in kernel density estimation.  We give two practical data driven approaches
to choose the bandwidth and demonstrate the performance through
simulations.

\subsection{Related work}

Our main technique for constructing prediction regions
is inspired by the {\em conformal prediction} method
\citep{VovkGS08b,ShaferV08}, a general approach
for constructing
distribution free, sequential prediction regions using
exchangeability.
Although in its original appearance,
conformal prediction is applied
to sequential classification and regression problems \citep{VovkNG09},
it is easy to adapt the method
to the prediction task described
in (\ref{eq:valid_exp}).
We describe this general method in Section \ref{sec::vovk} and our
adaptation in Section \ref{sec::kernel}.

In multivariate prediction region estimation, common approaches include
methods based on statistical equivalent blocks \citep{Tukey47,LiL08}
and plug-in density level sets \citep{ChatterjeeP80,Cadre06}.  In methods based on
statistical equivalent blocks, an ordering function taking values in
$\mathbb{R}^1$
is defined and used to order the data points.  Then one-dimensional
tolerance interval methods \citep[e.g.][]{Wilks41} can be applied.  Such
methods usually give accurate coverage but the efficiency is hard
to prove.  In particular, \citet{LiL08} proposed an estimator
using
the multivariate spacing depth as the ordering function.  Such a method
is completely nonparametric, requiring no tuning parameter,
and is adaptive to the shape of
the underlying distribution if the density level sets are convex.
However, this method requires $O(n^{d+1})$ time to compute the indicator
$\mathbf{1}(y\in C_n)$ for any given $y$, which is much
higher comparing to methods based on
plug-in density level sets.  Moreover, it is not clear how this method
performs when the level sets of underlying distribution are not convex.
On the other hand, the methods based on plug-in density level sets
\citep{ChatterjeeP80} gives provable validity and efficiency in
asymptotic sense regardless of the shape of the distribution
\citep{Cadre06}, while requiring only $O(n)$ time to compute the
indicator function.  The potential of such estimators has been reported
empirically in \citet{DiBucchianicoEM01}:
`` ... in principle the method based on density estimation  can perform very
well if a proper bandwidth is chosen, ...''

Our approach, although
originally inspired by conformal prediction, can be viewed as a combination
of the ordering based method and the density based method, where
the ordering function is given by the estimated density.  This agrees
with the simple fact that the best ordering function is just the
density itself.  To the best of our knowledge, this method is the first one with
both finite sample validity and explicit convergence rates.

There are other methods for multivariate prediction regions.  For example,
\citet{DiBucchianicoEM01} proposed to minimize the volume over a
pre-specified class of sets while maintaining a minimum coverage
under the empirical distribution.  This method works well for common
distributions whose level sets can be well approximated by regular
shapes such as ellipsoids and rectangles.
However, its performance depends crucially on the pre-specified
sets which cannot be very rich (must be a Donsker class), and hence cannot
be guaranteed for arbitrary distributions.  Moreover, the minimization
problem may be non-convex and hence computationally intensive.

The rest of this paper is organized as follows.
In Section \ref{sec::vovk}
we introduce conformal prediction.  In Section \ref{sec::kernel}
we describe
a construction of prediction region
by combining conformal prediction with kernel density estimator.
The approximation result (sandwiching lemma) and asymptotic
properties are
also discussed in Section \ref{sec::kernel}.
Practical methods for choosing the bandwidth
are given in Section \ref{sec::bandwidth}
and
simulation results are presented in Section \ref{sec:numerical}.  Some
closing remarks and possible future works are given in Section \ref{sec:conclusion}.
Some technical proofs are given in Section \ref{sec:proof}.

\section{Conformal prediction}
\label{sec::vovk}

We can construct a valid prediction region
using a method from
\citet{VovkGS08b} and \citet{ShaferV08}.
Although their focus was on sequential prediction with covariates, the same
basic idea can be used here.
The method is simple: consider a ``conformity measure''
$\sigma(P,y)$, which measures the ``conformity'' or ``agreement'' of
a point $y$ with respect to a distribution $P$.  Examples of such a
function in the multivariate case include data depth
\citep[see][and references therein]{LiuPS99}, and the density
function. For other
choices of conformity measure, see the book by \citet{VovkGS08b}. Given
an independent sample $Y_1,...,Y_n$ from $P$,
we test the hypothesis that $(Y_1,...,Y_n,Y_{n+1})\stackrel{iid}{\sim}P$
using observation $(Y_1,...,Y_n,y)$
for each $y\in \mathbb{R}^d$ and invert the test.
The test statistic is constructed using
$\sigma$ with $P$ replaced by empirical distribution $\hat P$.

When $(Y_1,...,Y_n,Y_{n+1})$ is a random sample from $P$, let
$\hat{P}_{n+1}$ be the corresponding empirical distribution,
which is symmetric
in the $n+1$ arguments.
Let
$$
\pi_{n+1,i}=
\frac{1}{n+1}\sum_{j=1}^{n+1}\mathbf{1}\left[
\sigma(\hat P_{n+1},Y_j)\le \sigma(\hat P_{n+1},Y_i)\right]\, .
$$
By
symmetry, the sequence of random variables
$\big(\sigma(\hat P_{n+1},Y_i):1\le i\le n+1\big)$ are exchangeable and hence so
are $(\pi_{n+1,i}:1\le i\le n+1)$.
Let
$$
\tilde\alpha = \frac{\lfloor (n+1)\alpha\rfloor}{n+1}.
$$
Note that
$(1 + 1/n)^{-1} \alpha \leq \tilde\alpha \leq \alpha$ and so
$\tilde\alpha \approx \alpha$.
Then,
for any
$\alpha\in (0,1)$,
\begin{equation}\label{eq:basic_conformal}
\mathbb{P}(\pi_{n+1,i}\ge \tilde\alpha)\ge 1-\alpha\,,
\end{equation}
since there are at least $(1-\alpha)(n+1)$ such $\pi_{n+1,i}$'s
satisfying $\pi_{n+1,i}\ge \tilde\alpha$.

Let
\begin{equation}\label{eq:conformal_set_generic}
\hat C^{(\alpha)}(Y_1,...,Y_n)=
\left\{y: \left(\left.\pi_{n+1,n+1}\right|_{Y_{n+1}=y}\right)\ge \tilde\alpha \right\}\,,
\end{equation}
where $\left.\pi_{n+1,n+1}\right|_{Y_{n+1}=y}$ is the random variable
$\pi_{n+1,n+1}$ evaluated at $Y_{n+1}=y$.
Then (\ref{eq:basic_conformal}) implies that
$$
\mathbb{P}\left(Y_{n+1}\in \hat C^{(\alpha)}(Y_1,...,Y_n)\right)\ge 1-\alpha.
$$
Based on the above discussion, any conformity measure $\sigma$ can be used
to construct prediction regions with finite sample validity, with
essentially no assumptions on $P$.  The only requirement is exchangeability
of $\{\pi_{n+1,i}\}$ which is satisfied if the sample is independent.

In this paper we use
\begin{equation}
\sigma(\hat P,y)=\hat{p}(y),
\end{equation}
that is,
a density estimate evaluated
at $y$.  We show that such a choice is closely related to
the plug-in density level set estimator and hence can be proved
to be asymptotically minimal with explicit rate of convergence.


\section{Kernel density estimation}
\label{sec::kernel}

Let $\mathbf Y=(Y_1,\ldots, Y_n)$.
Define the augmented data
${\sf aug}(\mathbf Y,y) = (Y_1,\ldots, Y_n,y)$.
Let $\hat p$ be some density estimator that is
defined for all $n$.
For example, $\hat p$ could be a parametric estimator
or a nonparametric estimator such as a kernel density estimator.
The particular algorithm we propose is given in Figure \ref{fig::algorithm}.

\begin{figure}
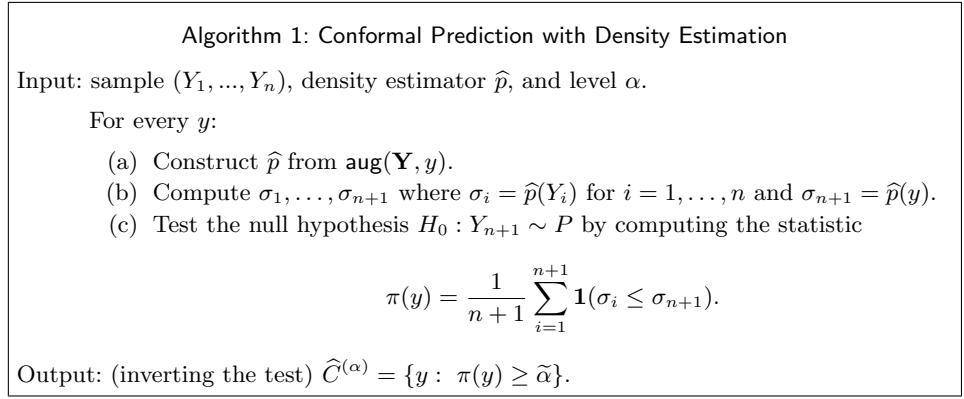

\fbox{\parbox{4.9in}{
\begin{center}{\sf Algorithm 1: Conformal Prediction with Density Estimation}\end{center}
Input: sample $(Y_1,...,Y_n)$, density estimator $\hat p$, and level $\alpha$.
\begin{enum}
\item [] For every $y$:
\begin{enum}
\item Construct $\hat{p}$ from ${\sf aug}(\mathbf Y,y)$.
\item Compute $\sigma_1,\ldots, \sigma_{n+1}$ where
$\sigma_i = \hat{p}(Y_i)$ for $i=1,\ldots, n$ and
$\sigma_{n+1} = \hat p(y)$.
\item Test the null hypothesis
$H_0: Y_{n+1} \sim P$ by computing the statistic
$$ \pi(y) = \frac{1}{n+1}\sum_{i=1}^{n+1} \mathbf{1}(\sigma_i
\leq \sigma_{n+1}). $$
\end{enum}
\end{enum}
Output: (inverting the test)
$\hat{C}^{(\alpha)}= \{ y:\ \pi(y) \ge \tilde \alpha\}$.
}}
\caption{The algorithm for computing the prediction region.}
\label{fig::algorithm}
\end{figure}

Recall that under the null hypothesis
$H_0: (Y_1,...,Y_n,Y_{n+1})\stackrel{iid}{\sim} P$,
the ranks of $\hat{p}(Y_i)$ are exchangeable, and hence
$\mathbb{P}\left(\pi(y) < \tilde\alpha\right) \le \alpha$.
Hence, we have:

\begin{lemma}
Suppose $Y_1,...,Y_{n},Y_{n+1}$ is an independent random sample
from $P$, then
\begin{equation}
\mathbb{P}\left(Y_{n+1}\in \hat C^{(\alpha)}\right) \ge 1-\alpha\,,
\end{equation}
for all probability measures $P$ and hence $\hat C^{(\alpha)}$ is valid.
\end{lemma}

\begin{remark}
Note that the prediction region is valid
(has correct finite sample coverage)
without any smoothness assumptions on $p$.
Indeed, the region is valid even if $P$ does not
have a density.
\end{remark}

\subsection{Conformal prediction with kernel density estimation}
Now we turn to the combination of conformal prediction with kernel density
estimator.
For a given bandwidth $h_n$ and kernel function $K$, let
\begin{equation}\label{eq:kernel_density}
\hat{p}_n(u) = \frac{1}{n}\sum_{i=1}^n \frac{1}{h_n^d}\, K\left(\frac{u-Y_i}{h_n}\right)
\end{equation}
be the usual kernel density estimator.
For now, we focus on a given bandwidth $h_n$.  The theoretical
and practical aspects of choosing $h_n$ will be discussed in Subsection
\ref{sec:asymp} and Section \ref{sec::bandwidth}, respectively.
For any given $y\in\mathbb{R}^d$, let
$Y_{n+1}=y$ and define
the augmented density estimator
\begin{align}
\hat{p}_n^y(u) =& \frac{1}{h_n^d(n+1)} \sum_{i=1}^{n+1} K\left(\frac{u-Y_i}{h_n}\right)
\nonumber\\
 =&
\left(\frac{n}{n+1}\right)\hat p_n(u) + \frac{1}{h_n^d(n+1)}
K\left(\frac{u-y}{h_n}\right).
\end{align}
Now we use the conformity
measure $\sigma(\hat P_{n+1}, Y_i)=\hat{p}_n^y(Y_i)$ and
the p-value is
$$
\pi(y) = \frac{1}{n+1}\sum_{i=1}^{n+1} \mathbf{1}
\left( \hat p^y_n (Y_i) \leq \hat p^y_n(y)\right).
$$
The resulting prediction region given by Algorithm 1 is
$\hat{C}^{(\alpha)} = \{ y:\ \pi(y) \ge \tilde\alpha\}$.

Figure \ref{fig::one-dim} shows a one-dimensional example
of the procedure, which we will investigate in detail later.
The top left plot shows a histogram of some data of sample size
20 from a two-component Gaussian mixture.
The next three plots
(top right, middle left, middle right) show three kernel
density estimators with increasing bandwidth as well as the
conformal prediction regions derived from these estimators
with $\alpha=0.05$.
Every bandwidth leads to
a valid region, but undersmoothing and oversmoothing lead to larger regions.
The bottom left plot shows the Lebesgue measure of the region as a function
of bandwidth. The bottom right plot shows the estimator and prediction region
based on the bandwidth whose corresponding conformal prediction region has
the minimal Lebesque measure.

\begin{figure}
\begin{center}
\includegraphics[scale=.7]{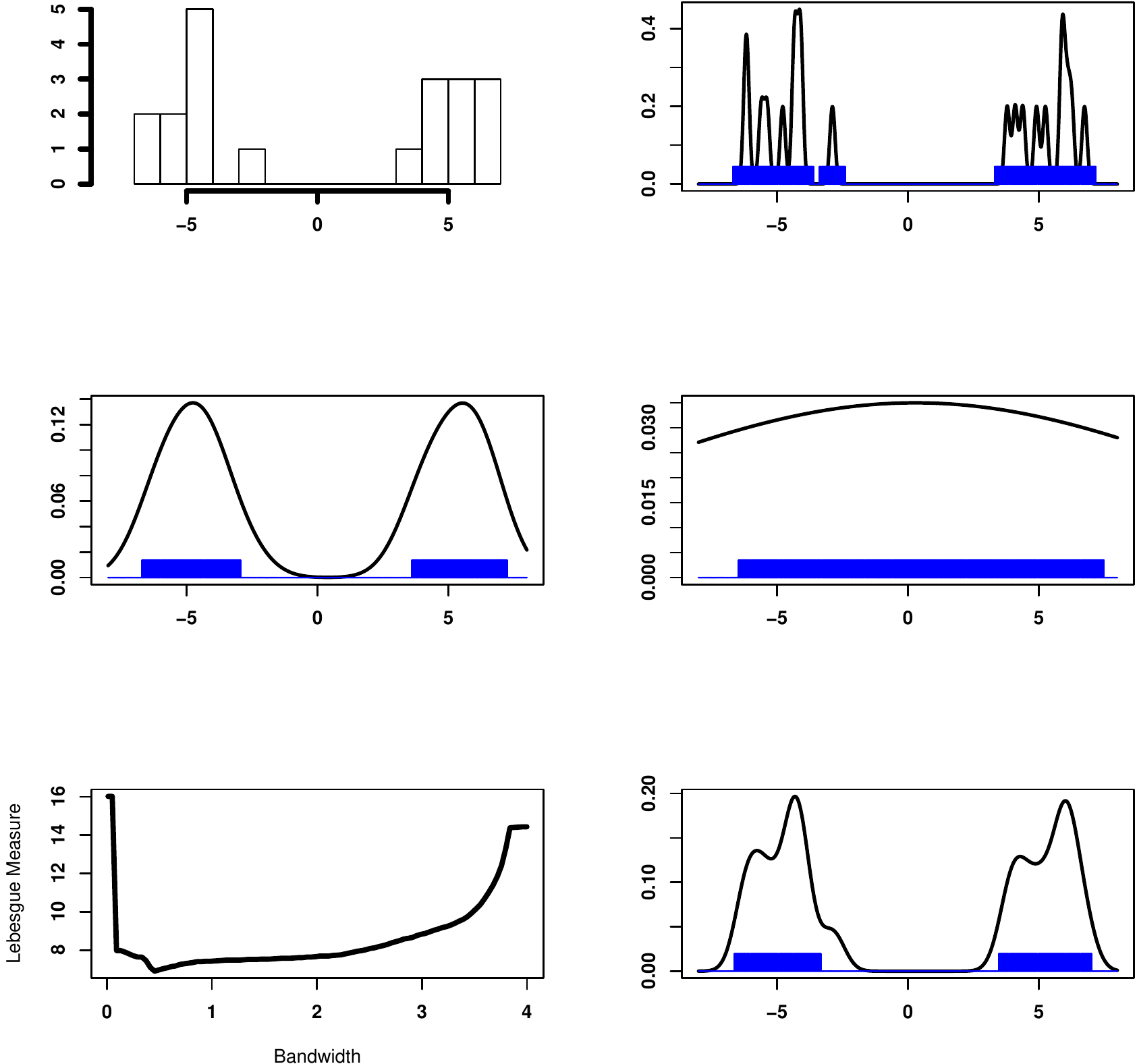}
\end{center}
\caption{Top left: histogram of some data.
Top right, middle left, and middle right show three kernel
density estimators with increasing bandwidth as well as the
conformal prediction regions derived from these estimators.
Bottom left: Lebesgue measure as a function
of bandwidth. Bottom right:  estimator and prediction region from
the bandwidth with smallest prediction region.}
\label{fig::one-dim}
\end{figure}

\subsection{An approximation}\label{subsec:approx}
The conformal prediction
region given by Algorithm 1 is closely related to the kernel density
estimator. In this subsection we further investigate this
connection and state the main approximation result,
the sandwiching lemma, which provides
simple characterization of the conformal prediction region in
terms of
plug-in kernel density level sets.  The sandwiching lemma will also
be useful in the study of efficiency of the conformal prediction regions.

We first introduce some notation.
Define
the upper and lower level sets of density $p$ at level $t$,
respectively:
\begin{equation}
L(t)=\{y:p(y)\ge t\},\quad{\rm and}\quad L^{\ell}(t)=\{y:p(y)\le t\}.
\end{equation}
The corresponding level sets of $\hat p_n$ are denoted $L_n(t)$ and $L_n^\ell(t)$,
respectively.
Let
\begin{equation}
P_n^y=\frac{n}{n+1}P_n+\frac{1}{n+1}\delta_y,
\end{equation}
where $P_n$ is the empirical distribution defined by the sample
$\mathbf Y=(Y_1,...,Y_n)$,
and $\delta_y$ is the point mass distribution at $y$.
Define functions
\begin{align}
G(t)&=P(L^\ell(t)),\nonumber\\
G_n(t)&=P_n(L_n^\ell(t))=n^{-1}\sum_{i=1}^n
\mathbf{1}(\hat{p}_n(Y_i)\le t),\nonumber\\
G^y_n(t)&=P_n^y(\hat{p}^y_n(Y)\le t)\nonumber\\
&=(n+1)^{-1}
\left(\sum_{i=1}^{n} \mathbf{1}(\hat{p}_n^y(Y_i)\le t)+
\mathbf{1}(\hat{p}_n^y(y)\le t)\right).\nonumber
\end{align}
The functions $G$, $G_n$ and $G_n^y$ defined above are the cumulative
distribution function (CDF) of $p(Y)$ and its empirical versions
with sample $\mathbf Y$ and ${\sf aug}(\mathbf Y,y)$, respectively.

By (\ref{eq:conformal_set_generic}) and Algorithm 1,
the conformal prediction region can be written as
\begin{equation}
\hat{C}^{(\alpha)}=\Bigl\{y\in \mathbb{R}^d: G_n^y(\hat{p}_n^y(y))\ge \tilde\alpha\Bigr\}.
\end{equation}
Let $Y_{(1)},\dots,Y_{(n)}$ be the reordered data so that
$\hat{p}_n(Y_{(1)}),\dots,\hat{p}_n(Y_{(n)})$ are in ascending order.
Let $i_{n,\alpha}= \lfloor (n+1)\alpha\rfloor$, and define
the inner and outer sandwiching sets:
$$
L_n^-=L_n\left(\hat{p}_n(Y_{(i_{n,\alpha})})\right)
$$
and
$$
L_n^+=L_n\left(\hat{p}_n(Y_{(i_{n,\alpha})})-(nh^d)^{-1}\psi_K\right)\,,
$$
where
$\psi_K =\sup_{u,u'}|K(u)-K(u')|$.
Then we have the following ``sandwiching''
lemma, whose proof can be found in
Subsection \ref{subsec:proof_sandwich}.

\begin{lemma}[Sandwiching Lemma]   \label{lem:sandwich}
Assume that $||K||_\infty=K(0)$, then
\begin{equation}
L_n^-
\subseteq \hat{C}^{(\alpha)} \subseteq
L_n^+\,.
\end{equation}
\end{lemma}

According to the sandwiching lemma, $L_n^+$ also guarantees
distribution free finite sample coverage
and it is easier to analyze.
The inner region, $L_n^-$, which is not much smaller than $L_n^+$ when $n$
is large, generally does not have finite sample validity.  We confirm
this through simulations in Section \ref{sec:numerical}.
Next we investigate the efficiency of these prediction regions.

\subsection{Asymptotic properties}\label{sec:asymp}
In this subsection we prove asymptotic efficiency of $\hat C^{(\alpha)}$ and the sandwiching sets
in terms of the convergence rates of their loss.

Recall that the optimal prediction region at level $1-\alpha$ can be
written as
\begin{equation}
C^{(\alpha)}=L(t^{(\alpha)})=\{y: p(y)\ge t^{(\alpha)}\},
\end{equation}
where $t^{(\alpha)}$ is the cut-off value of the
density function so that the probability mass in the lower level set is
exactly $\alpha$:
\begin{equation}
G(t^{(\alpha)})=\mathbb{P}(p(Y)\le t^{(\alpha)})=\alpha.
\end{equation}
This holds if we
assume $G$ is continuous at $t^{(\alpha)}$ so that the above equation
implies $\mathbb{P}(p(Y)\ge t^{(\alpha)})=1-\alpha$.  This is equivalent
to assuming that the contour of $p$ at value $t^{(\alpha)}$,
$\{y:p(y)=t^{(\alpha)}\}$, has zero measure
under $P$.

The inner and outer sandwiching sets $L_n^-$ and $L_n^+$
are plug-in estimators of density level sets of the form:
\begin{equation}
L_n(t_n^{(\alpha)})=\{y:\hat{p}_n(y)\ge t_n^{(\alpha)}\},
\end{equation}
where
$t_n^{(\alpha)}=\hat p_n(Y_{(i_{n,\alpha})})$ for the inner set $L_n^-$ and
$t_n^{(\alpha)}=\hat p_n(Y_{(i_{n,\alpha})})-(nh_n^d))^{-1}\psi_K$
for the outer set $L_n^+$.
Here
we can view $t_n^{(\alpha)}$ as an estimate
of $t^{(\alpha)}$.  In \cite{CadrePP09} it is shown that, under
regularity conditions of the density $p$, the plug-in estimators
$t_n^{(\alpha)}$ and $L_n(t_n^{(\alpha)})$ using kernel density
estimator are consistent with convergence rate $1/\sqrt{nh_n^d}$
for a range of $h_n$.  Here, we refine the results
using a set of slightly modified conditions.

Intuitively speaking, for any density estimator $\hat p_n$
and cut-off values $t_n^{(\alpha)}$, the plug-in density level set
$L_n(t_n^{(\alpha)})$ is an accurate estimator of $L(t^{(\alpha)})$
 if:
\begin{enumerate}
  \item The estimated density function, $\hat p_n$, is close to the
  true density $p$.
  \item The true density is not too flat around level $t^{(\alpha)}$.
  \item The estimated cut-off value $t_n^{(\alpha)}$ is an accurate estimate of
  $t^{(\alpha)}$.
\end{enumerate}
The first condition has been extensively studied in the literature
of nonparametric density estimation and sufficient conditions of
convergence for kernel density estimators in various forms have been
established.  The second condition is more specific for density level
set estimation.  A common condition is the $\gamma$-exponent at
level $t^{(\alpha)}$, which is first introduced by \cite{Polonik95}
and has been used by many others
\citep[see][for example]{Tsybakov97,RigolletV09}.  The third condition is
somewhat opposite to the second one.  It essentially requires that
the density function cannot be too steep near the true cut-off value.
This turns out to be a natural condition whenever the density has
bounded derivatives near the contour.  We formalize this condition
through a ``modified $\gamma$-exponent condition'' which is detailed
in Section \ref{sec:gamma_expo}.

\subsubsection{H\"{o}lder Classes of Densities}

To study the efficiency of the prediction region,
we need some smoothness
condition on $p$.  The H\"{o}lder class is a popular smoothness condition in nonparametric
inferences \citep[Section 1.2]{Tsybakov09}.  Here we use the version
given in \cite{RigolletV09}. 

Let $s=(s_1,...,s_d)$ be a $d$-tuple of non-negative integers and
$|s|=s_1+...+s_d$.  For any $x\in \mathbb{R}^d$, let
$x^s=x_1^{s_1}\cdots x_d^{s_d}$ and $D^s$ be the differential
operator:
 $$D^s f=\frac{\partial^{|s|}f}{\partial x_1^{s_1}\cdots
 \partial x_d^{s_d}}(x_1,...,x_d).$$  Given $\beta>0$,
 for any functions $f$ that are $\lfloor\beta\rfloor$ times
 differentiable, denote its Taylor expansion of degree
 $\lfloor\beta\rfloor$ at $x_0$ by
 $$f^{(\beta)}_{x_0}(x)=\sum_{|s|\le \beta}\frac{(x-x_0)^s}{s_1!\cdots s_d!}D^sf(x_0).$$

\begin{definition}[H\"{o}lder class]
For constants $\beta >0$, $L>0$, define the H\"{o}lder class
$\Sigma(\beta, L)$ to be the set of $\lfloor \beta\rfloor$-times
differentiable functions on $\mathbb{R}^d$ such that,
\begin{equation}
|f(x)-f^{(\beta)}_{x_0}(x)|\le L||x-x_0||^{\beta}.
\end{equation}
\end{definition}

\subsubsection{The $\gamma$-exponent condition}\label{sec:gamma_expo}

For a density function $p$, and a level $t\in (0,||p||_\infty)$, the
usual $\gamma$-exponent condition requires that there exists an
$\epsilon_0>0$ and $c_1>0$ such that
\begin{equation}\label{eq:gamma_exp}
\mu(\{y:t<p(y)\le t+\epsilon\})\le c_1\epsilon^\gamma,~\forall
\epsilon\in(0,\epsilon_0).\end{equation}
Condition (\ref{eq:gamma_exp}) is essentially requiring that the density
$p(y)$ increases roughly at rate $\epsilon^{1/\gamma}$ when $y$ moves
away from the contour by an $\epsilon$ distance.
As a result, a larger value of $\gamma$
corresponds to a faster change of the density $p$ when moving away
from the contour, hence it is easier to estimate the density level set.
In this paper, we consider the modified $\gamma$-exponent condition:

\begin{definition}[Modified $\gamma$-exponent condition]\label{def:modif_gamma}
We say a density function $p$ satisfies the modified $\gamma$-exponent
condition at level $t$, if there exist constants $\epsilon_0>0$ and
$c_1,c_2>0$, such that
\begin{equation}\label{eq:modif_gamma_exp}
c_1\le
\frac{P(\{y:t_-\le p(y)\le t_+\})}{(t_+-t_-)^\gamma}\le c_2,
~~\forall ~t-\epsilon_0\le t_-<t_+\le t+\epsilon_0.\end{equation}
\end{definition}

The modified $\gamma$-exponent condition differs from the original definition in three aspects:

\begin{enumerate}\item First, it allows both sides of the interval to
change within a neighborhood of $t$, which is stronger than
(\ref{eq:gamma_exp}). It does not allow the contour at level $t$ to
have positive measure.  We note that if the contour at level  $t$ has
positive measure, then the estimated level set has at least a constant 
loss unless the cut-off value is estimated without error.
 \item Second, it does not only require an upper bound on the measure,
 but also a lower bound.  Since the upper bound indicates that the
 density cannot be too flat around the contour, the lower bound does
 not allow the density to be too steep. This condition implies that the
 estimated cut-off value is close to the truth.  It usually holds when
 the density is smooth enough around the contour.  For example, when the
 contour at level $t$ is smooth and the density $p$ satisfies
 $|p(y)-t|\approx \delta^{1/\gamma}$ for all $y$ that is $\delta$ away from
 the contour
 and all $\delta$ small enough \citep{Tsybakov97}.
 \item Moreover, in the modified condition, we use the measure induced
 by $p$, rather than the Lebesgue measure.  This is a minor difference
 since we always have, for all $t-\epsilon_0\le t_-<t_+\le t+\epsilon_0$,
$$t-\epsilon_0\le\frac{P(\{y:t_-<p(y)\le t_+\})}
{\mu(\{y:t_-<p(y)\le t_+\})}\le t+\epsilon_0.$$
\end{enumerate}

\subsubsection{Conditions on the Kernel}

A standard condition on the kernel
is the notion of $\beta$-valid kernels.
\begin{definition}[$\beta$-valid kernel]
  \label{def:valid_kernel} For any $\beta>0$,
  a function $K:\mathbb{R}^d\mapsto \mathbb{R}^1$ is a $\beta$-valid
  kernel if
  \begin{enumerate} \item $K$ is supported on $[-1,1]^d$.
  \item $\int K=1$.
  \item $\int |K|^r <\infty$, all $r\ge 1$.
  \item $\int y^s K(y)dy=0$ for all $1\le |s|\le \beta$.
  \end{enumerate}
\end{definition}
In the literature, $\beta$-valid kernels are usually used with H\"{o}lder
class of functions to derive fast rate of convergence.  The existence of
univariate $\beta$-valid kernels can be found in
\cite[Section 1.2]{Tsybakov09}.  A multivariate $\beta$-valid kernel can
be obtained by taking direct product of univariate $\beta$-valid kernels.

\subsubsection{Asymptotic properties of estimated density level set}

Consider the following assumptions:

\vspace{0.2cm}

\noindent\textbf{Assumption A1:}
\begin{enum}
  \item [(a)] The density function $p\in \mathcal P(\beta, L)$, where
  $\mathcal P(\beta,L)$ is the class of all density functions that are
  in the H\"{o}lder class $\Sigma(\beta,L)$.
  \item [(b)] The density $p$ satisfies the modified $\gamma$-exponent
  condition at level $t^{(\alpha)}$.
  \item [(c)] The density function $p$ is uniformly bounded by a
  constant $\bar{L}$.
\end{enum}

\noindent\textbf{Assumption A2:}
The bandwidth satisfies
\begin{equation}
h_n\asymp \left(\frac{\log n}{n}\right)^{\frac{1}{2\beta+d}}.
\end{equation}

\noindent\textbf{Assumption A3:}
  The kernel $K$ is $\beta$-valid and $||K||_\infty=K(0)$.

\vspace{0.5cm}

These assumptions extend those in
\citep{CadrePP09},
where $\beta=1$ is considered.  Also A1(b)
considered here is a local version.

The next theorem states the quality of
cut-off values used in the sandwiching sets
$L_n^-$ and $L_n^+$.
\begin{theorem}\label{thm:rate_of_level}
Let $t_n^{(\alpha)}=\hat p_n(Y_{(i_{n,\alpha})})$, where $\hat p_n$
is the kernel density estimator given by eq. (\ref{eq:kernel_density}), and $Y_{(i)}$ and $i_{n,\alpha}$ are defined
as in Section \ref{subsec:approx}.
Under assumptions A1-A3, for any $\lambda>0$, there
exist constants $A_\lambda$, $A_\lambda'$ depending
only on $p$, $K$ and $\alpha$, such that
\begin{equation}\label{eq:rate_of_level}
\mathbb{P}\left(|t_n^{(\alpha)}-t^{(\alpha)}|\ge
A_\lambda \left(\frac{\log n}{n}\right)^{\frac{\beta}{2\beta+d}}
+ A_\lambda'
\left(\frac{\log n}{n}\right)^{\frac{1}{2\gamma}}
\right)=O(n^{-\lambda}).
\end{equation}
\end{theorem}

We give the proof of Theorem
\ref{thm:rate_of_level} in Section \ref{subsec:proof_rate_of_level}.
Theorem \ref{thm:rate_of_level} is
useful for establishing the convergence of the
corresponding level set.
Observing that $(nh_n^d)^{-1}=o((\log n/n)^{\beta/(2\beta+d)})$,
it follows
immediately that the cut-off value used in $L_n^+$
also satisfies (\ref{eq:rate_of_level}).
The next theorem gives the rate of convergence for
plug-in level set estimators when the cut-off value
satisfies (\ref{eq:rate_of_level}).

\begin{theorem}  \label{thm:L_n+}
Let $t_n^{(\alpha)}$ be a random sequence which satisfies (\ref{eq:rate_of_level}).
Under A1-A3, for any $\lambda>0$, there exist constants $B_\lambda$, $B_\lambda'$
depending on $p$, $K$ and $\alpha$ only, such that
\begin{align}
\mathbb{P}\left(\mu(L_n(t_n^{(\alpha)})\triangle C^{(\alpha)})\ge
B_\lambda\left(\frac{\log n}{n}\right)^{\frac{\beta\gamma}{2\beta+d}}
+B_\lambda'\left(\frac{\log n}{n}\right)^{\frac{1}{2}}\right)
=O(n^{-\lambda})\,.\nonumber
\end{align}
\end{theorem}

By Theorem \ref{thm:rate_of_level}, the cut-off values
used in  $L_n^-$ and $L_n^+$ both satisfy
(\ref{eq:rate_of_level}), so the convergence rate in
Theorem \ref{thm:L_n+} holds
for $L_n^-$ and $L_n^+$.
By Lemma \ref{lem:sandwich}, it also holds for $\hat C^{(\alpha)}$.

\begin{corollary}
\label{cor:rates}
Under A1-A3, for any $\lambda>0$, there exists constant
$B_\lambda$, $B_\lambda'$
depending on $p$, $K$ and $\alpha$ only, such that, for
all $\hat C\in \{\hat C^{(\alpha)},L_n^-,L_n^+\}$,
\begin{equation}\label{eq:final_rate}
\mathbb{P}\left(\mu(\hat C\triangle C^{(\alpha)})\ge
B_\lambda\left(\frac{\log n}{n}\right)^{\frac{\beta\gamma}{2\beta+d}}
+B_\lambda'\left(\frac{\log n}{n}\right)^{\frac{1}{2}}\right)=O(n^{-\lambda}).
\end{equation}
\end{corollary}

In the most common case $\beta=\gamma=1$, the term
$(\log n /n)^{\beta\gamma/(2\beta+d)}$ dominates the convergence rate.
If we further assume that
the level
set $L(t^{(\alpha)})$ is star-shaped (or more generally,
a union of star-shaped sets), then the rate given by
Corollary \ref{cor:rates} is near optimal, up to a logarithm term.
Indeed, the rate in equation (\ref{eq:final_rate}) is within a logarithm
term of the minimax risk for density level set estimation as developed
in \cite{Tsybakov97}.  But note that the problem considered here is
harder than estimating density level set at a fixed level
since the cut-off value
is not known in advance and needs to be estimated.  Indeed, the
logarithm term comes from estimating $t_n^{(\alpha)}$.  We also note
that the continuity condition on $p$ is slightly different than that
in \cite{Tsybakov97} where it is assumed that the density contour at
the desired level is in a H\"{o}lder class.
But the same construction of the lower bound
can be used under the global smoothness conditions A1(a) and A1(b).

A minimax risk rate of the
plug-in density level set at a fixed level has been developed by
\cite{RigolletV09}.  Although the rate is similar as that obtained in
this paper, the construction of the lower bound only applies to fixed
cut-off values close to 1, and hence has only limited application to
the range of $\alpha$ values of practical interest.

\section{Choosing the bandwidth}
\label{sec::bandwidth}

As illustrated in Figure \ref{fig::one-dim}, the efficiency of $\hat
C^{(\alpha)}$ depends on the choice of $h_n$.  The size of estimated
prediction region can be very large if the bandwidth is either too
large or too small. Therefore, in practice it is desirable to choose
a good bandwidth in an
automatic and data driven manner.  In kernel density estimation,
the choice of bandwidth has been one of the most important topics and
many approaches have been studied; see \cite{Loader99} and \cite{MammenMNS11} and
references therein.
Intuitively, a good density estimator $\hat p$ will likely lead to
a good prediction region, and the dependence on $n$ of the (near)
optimal choice of $h_n$ in Theorem \ref{thm:L_n+} is similar to that in the context of kernel
density estimation.  However, this is not quite the case \citep{SamworthW10}. The
intuition is simple:
For density estimation, a good bandwidth guarantees the accuracy
of estimated density in the whole space, whereas for level sets it suffices
to estimate the density accurately near the contour.

We propose two
practical methods to choose a good bandwidth from a given candidate set
$\mathcal H = \{h_1,\ldots, h_m\}$, based on the idea that
a good prediction region has small Lebesgue measure;
see Figures \ref{fig::algorithm2} and \ref{fig::algorithm3}.
The methods introduced here are applicable to any
prediction region estimator $\hat C$ with finite sample validity.
In both approaches, we compute the prediction region for each $h\in {\cal H}$
and choose the one with the smallest volume.
To preserve finite sample validity,
the first approach, described in Fig \ref{fig::algorithm2}, uses a Bonferroni correction.

\begin{figure}
\fbox{\parbox{4.6in}{
\begin{center}{\sf Algorithm 2: Tuning with Bonferroni Correction}\end{center}
Input: sample $\mathbf Y=(Y_1,...,Y_n)$, prediction region estimator $\hat C$,
and level $\alpha$.
\begin{enum}
\item Construct
prediction sets
$\{\hat C_h=\hat C_{h}(Y_1,...,Y_n):\ h\in {\cal H}\}$
each at level $1-\alpha/m$, where $m=|\cal H|$.
\item Let $\hat h = \argmin_h \mu(\hat C_{h})$.
\item Return $\hat C_{\hat h}$.
\end{enum}
}}
\caption{Algorithm 2: bandwidth selection.}
\label{fig::algorithm2}
\end{figure}

\begin{proposition}
If $\hat C$ satisfies finite sample validity for any $h$, then the
estimated prediction region $\hat C_{\hat h}$ given by Algorithm 2 also
satisfies finite sample validity.
\end{proposition}
\begin{proof}
  Using Bonferroni correction we have \begin{align}
    \mathbb{P}(Y_{n+1}\in \hat C_{\hat h})\ge &
    \mathbb{P}(Y_{n+1}\in
    \hat C_h,~\forall h\in \cal H)\nonumber\\
    \ge& 1-\sum_{h\in \cal H}\mathbb{P}(Y_{n+1}\notin \hat C_h)\nonumber\\
    \ge& 1-\alpha\,,\nonumber
  \end{align}
where the last inequality uses the fact that each $\hat C_h$ is a
finite sample valid prediction region at level $1-\alpha/m$.
\end{proof}

When $m=|\mathcal{H}|$ is large, Algorithm 2 tends to be
conservative since each single $\hat C_h$
has coverage $1-\alpha/m$, which could be much bigger than the ideal
$(1-\alpha)$ region.
The algorithm described in Figure \ref{fig::algorithm3} uses sample splitting and only sacrifices
a constant rate of efficiency regardless of $|\mathcal{H}|$.

\begin{figure}
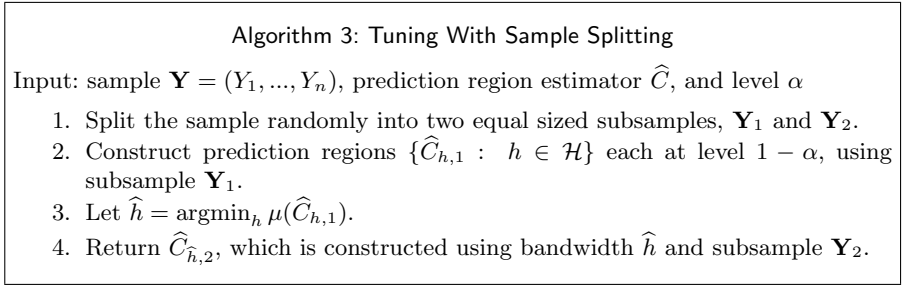

\fbox{\parbox{4.6in}{
\begin{center}{\sf Algorithm 3: Tuning With Sample Splitting}\end{center}
Input: sample $\mathbf Y=(Y_1,...,Y_n)$, prediction region
estimator $\hat C$, and level $\alpha$
\begin{enum}
\item Split the sample randomly into two equal sized subsamples, $\mathbf{Y}_1$ and $\mathbf{Y}_2$.
\item Construct
prediction regions
$\{\hat C_{h,1}:\ h\in {\cal H}\}$
each at level $1-\alpha$, using subsample
$\mathbf Y_1$.
\item Let $\hat h = \argmin_h \mu(\hat C_{h,1})$.
\item Return $\hat C_{\hat h,2}$, which is constructed using
bandwidth $\hat h$ and
subsample $\mathbf Y_2$.
\end{enum}
}}
\caption{Algorithm 3: bandwidth selection.}
\label{fig::algorithm3}
\end{figure}

\begin{proposition}
If $\hat C$ satisfies finite sample validity for all $h$, 
then $\hat C_{\hat h, 2}$, the output of Algorithm 3, also satisfies finite sample validity.
\end{proposition}
\begin{proof}
Note that $\hat h$ is independent of $\mathbf Y_2$, as a result,
\begin{align}
  \E \left(P\big(\hat C_{\hat h, 2}\big)\right)=&
  \E\left(\E\left(P\big(\hat C_{\hat h, 2}\big)\big|\hat h\right)\right)\nonumber\\
  \ge& \E\left(1-\alpha|\hat h\right)\nonumber\\
  =&1-\alpha\,.\nonumber
\end{align}
\end{proof}

It is easy to see that these methods
have small excess loss with high probability
since, by construction,
$\mu(\hat C)\le \mu(\hat C_{h^*})+\nu_n$,
where $h^*$ is the best bandwidth that minimizes
the excess loss $\mathcal E(\mu(\hat C_{h}))$ and
$\nu_n$ is a negligible term, because for
${\cal H}$ dense enough, there exists $h_j\in {\cal H}$
such that $h_j \approx h^*.$
Although minimizing excess loss does not necessarily
minimize the symmetric difference loss,
a small excess loss itself is a desired
property in practice and is also a necessary condition of
small symmetric difference loss.
However, a more detailed relationship between excess loss
and symmetric difference loss requires extra conditions
and we leave that for future research.

\section{Numerical example}\label{sec:numerical}
A simple illustration of Algorithm 1 is presented in Figure
\ref{fig::one-dim}.  Here we
consider a two-dimensional Gaussian mixture, whose geometric
structure allows
a better visualization of the results.  We also
test the 
bandwidth selectors presented in Section \ref{sec::bandwidth}.  Due to the small
value of $\alpha$ and limited sample size,
we find Algorithm 3 more preferable than
Algorithm 2.  Thus we only present the results using bandwidth
chosen by sample splitting. For example,
when $n=200$, 100 data points are used to select the bandwidth and
the other 100 data points are used to construct the prediction
region using the selected bandwidth.

Table \ref{table::results} shows the
coverage and Lebesgue measure of the prediction region of
level .90 over 1,000 repetitions. 
The coverage is excellent and the size of the region
is close to optimal.  Both the conformal region $\hat C^{(\alpha)}$
and the outer sandwiching set $L_n^+$ gives correct coverage
regardless of the sample size.  It is worth noting that
the inner sandwiching set $L_n^-$ does not give the desired coverage, which
suggests that decreasing the cut-off value in $L_n^+$ is not
merely an artifact of proof, but a necessary tuning.
The observed excess loss also reflects a rate
of convergence that supports our theoretical results
on the symmetric difference loss.
Taking
$\hat C^{(\alpha)}$ for example, in Corollary \ref{cor:rates}
we have $\beta=\gamma=1$, $d=2$, and
$$\frac{\sqrt{(\log 200) / 200}}{\sqrt{(\log 1000) / 1000}}\approx 1.9,$$
which agrees with the observed drop of average excess loss
from 6 to 3 as
$n$ increased from 200 to 1,000.

\begin{table*}
\caption{The simulation results for 2-d Gaussian mixture with
$\alpha=0.1$ over 1000 repetitions. The Lebesgue measure of the
ideal region $\approx 28.02$.}
\label{table::results}
\begin{tabular}
  {lllll}
  \hline
  &\multicolumn{2}{c}{Coverage}&\multicolumn{2}{c}{Lebesgue Measure}\\
  &\multicolumn{1}{c}{$n=200$}&\multicolumn{1}{c}{$n=1000$}&
  \multicolumn{1}{c}{$n=200$}&\multicolumn{1}{c}{$n=1000$}\\
  \hline
$\hat{C}^{(\alpha)}$&$0.897\pm0.002$&$0.900\pm 0.001$&$34.3\pm 0.31$&$31.1\pm0.15$\\
$L_n^-$&$0.882\pm0.001$&$0.896\pm 0.001$&$34.1\pm
0.22$&$32.2\pm0.10$\\
$L_n^+$&$0.900\pm0.001$&$0.907\pm 0.001$&$36.9\pm
0.21$&$34.1\pm0.10$\\
\hline
\end{tabular}
\end{table*}

Figure \ref{fig:2d_l_shape}
shows a typical realization of the estimators. In both panels,
the dots are data points when $n=200$.  The left panel shows the
conformal prediction region with sample splitting (blue curve),
together with the
inner and outer sandwiching sets (red and green curves, respectively).
Also plotted is the ideal region $C^{(\alpha)}$ (the grey curve).
It is clear that all three estimated regions captures the main
part of the ideal region, and they are mutually close.
On the right panel we plot a realization of the depth based approach
from \cite{LiL08}.  This approach does not require any tuning parameter.
However, it takes $O(n^{d+1})$ time to evaluate
$\mathbf 1 (y\in \hat C)$ for
any single $y$.  In practice it is recommended to compute the
empirical depth only for all the data points and use the convex hull of
all data points with high depth as the estimated prediction region. As
can be seen on the picture, such a convex hull construction misses the
``L'' shape of the ideal region.  Moreover, the kernel density method is
at least 1,000 times faster than the depth based method in
our implementation even when $n=200$.

\begin{figure}[t]
\begin{center}
\includegraphics[scale=0.35]{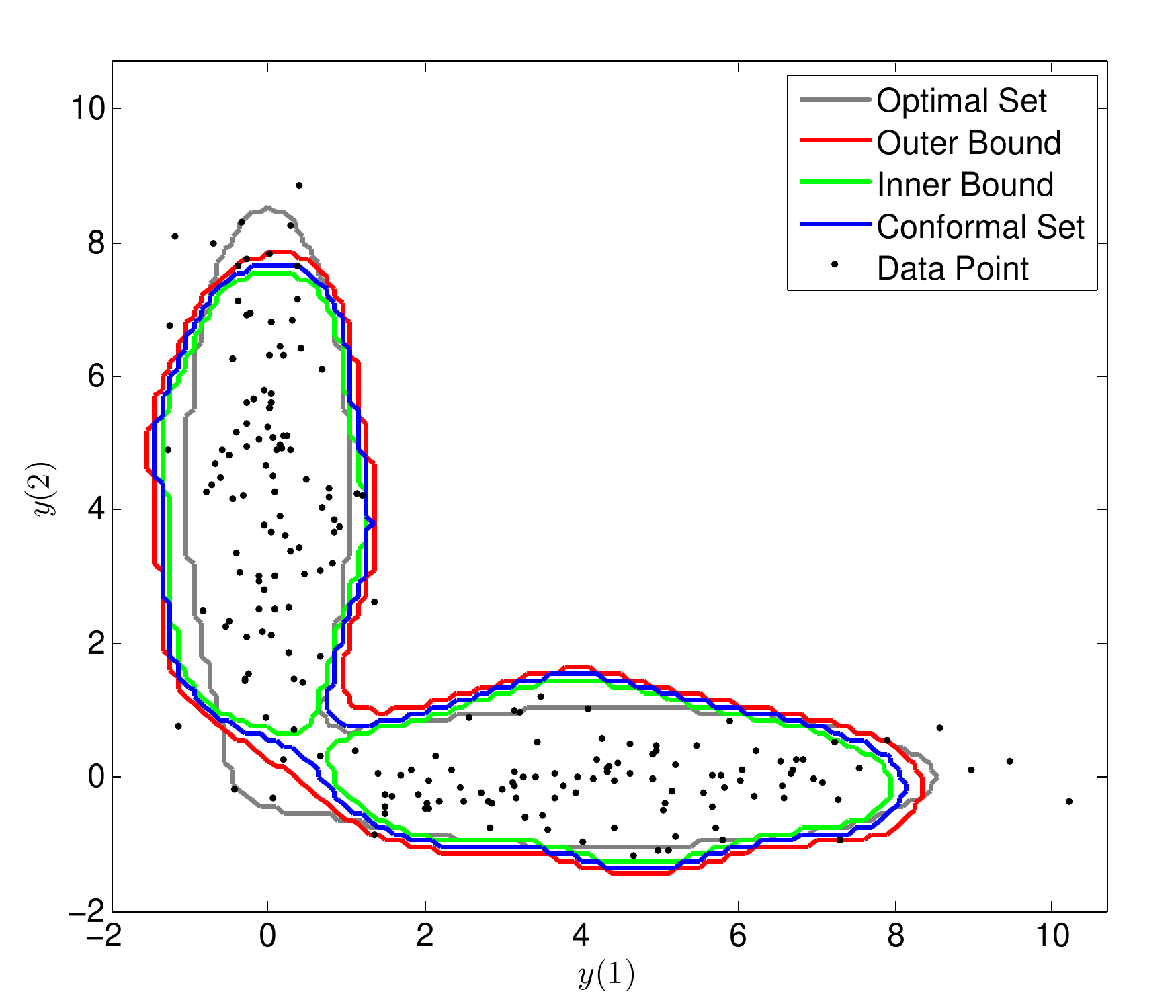}
\includegraphics[scale=0.35]{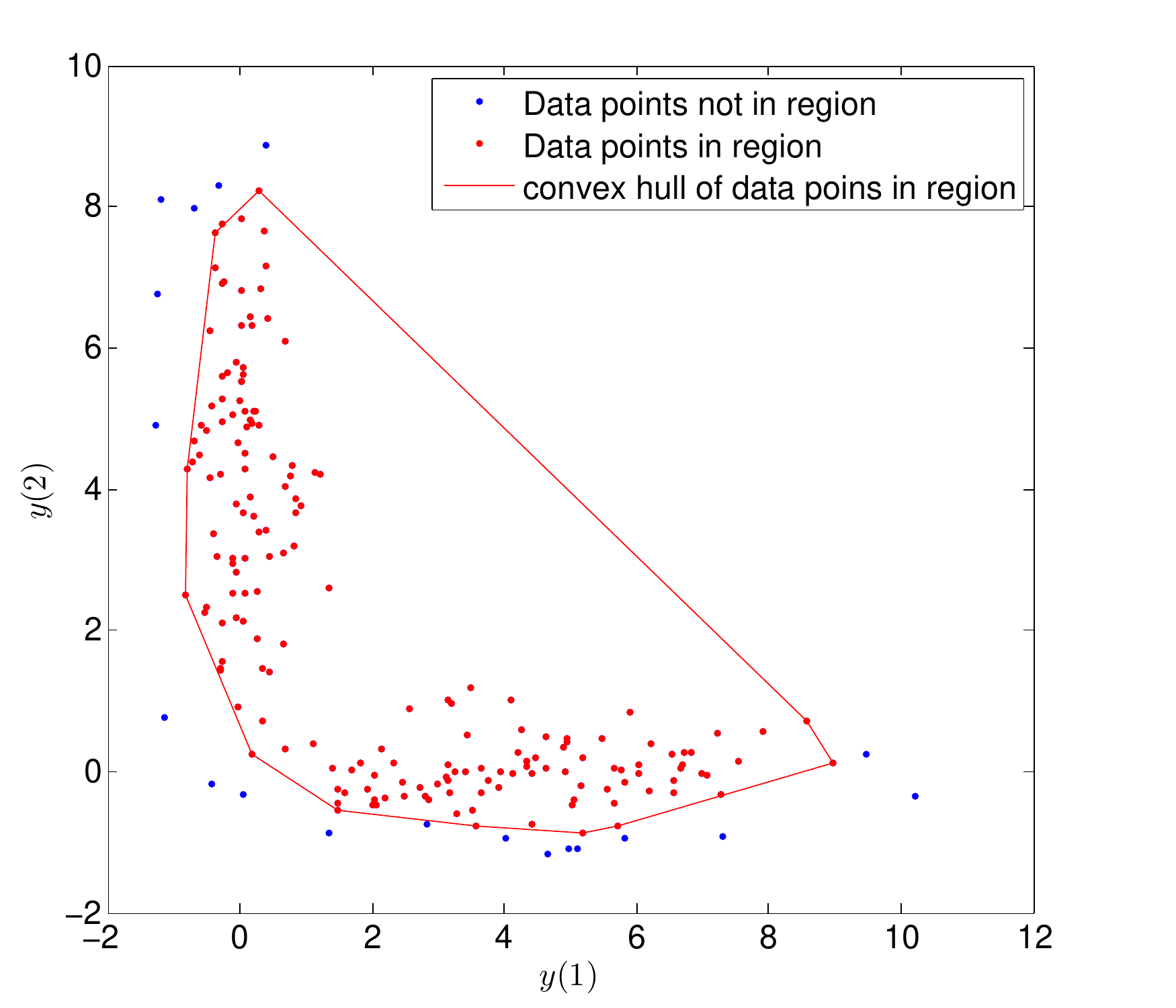}
\end{center}
\caption{Conformal prediction region
(left) and the convex hull of the multivariate spacing depth based
tolerance region (right), with data from a two-component Gaussian mixture.}
\label{fig:2d_l_shape}
\end{figure}

Figure \ref{fig:bandwidth} shows the effect of bandwidth on the
excess loss based on a typical implementation of conformal prediction,
where the $y$ axis is the Lebesgue measure of the
estimated region. We observe that for the conformal
prediction region $\hat C^{(\alpha)}$, the excess loss is
stable for a wide range of bandwidth, especially that moderate
undersmoothing does not harm the performance very much.
An intuitive explanation is that the data
near the contour is dense enough to allow for moderate
undersmoothing.  Similar phenomenon should be
expected whenever $\alpha$ is not too small.  Moreover,
the selected bandwidth from the outer sandwiching set $L_n^+$
is close to that obtained from the conformal region.  This
observation may be of practical interest since it is usually
much faster to compute $L_n^+$.

\begin{figure}[t]
\begin{center}
\includegraphics[scale=0.35]{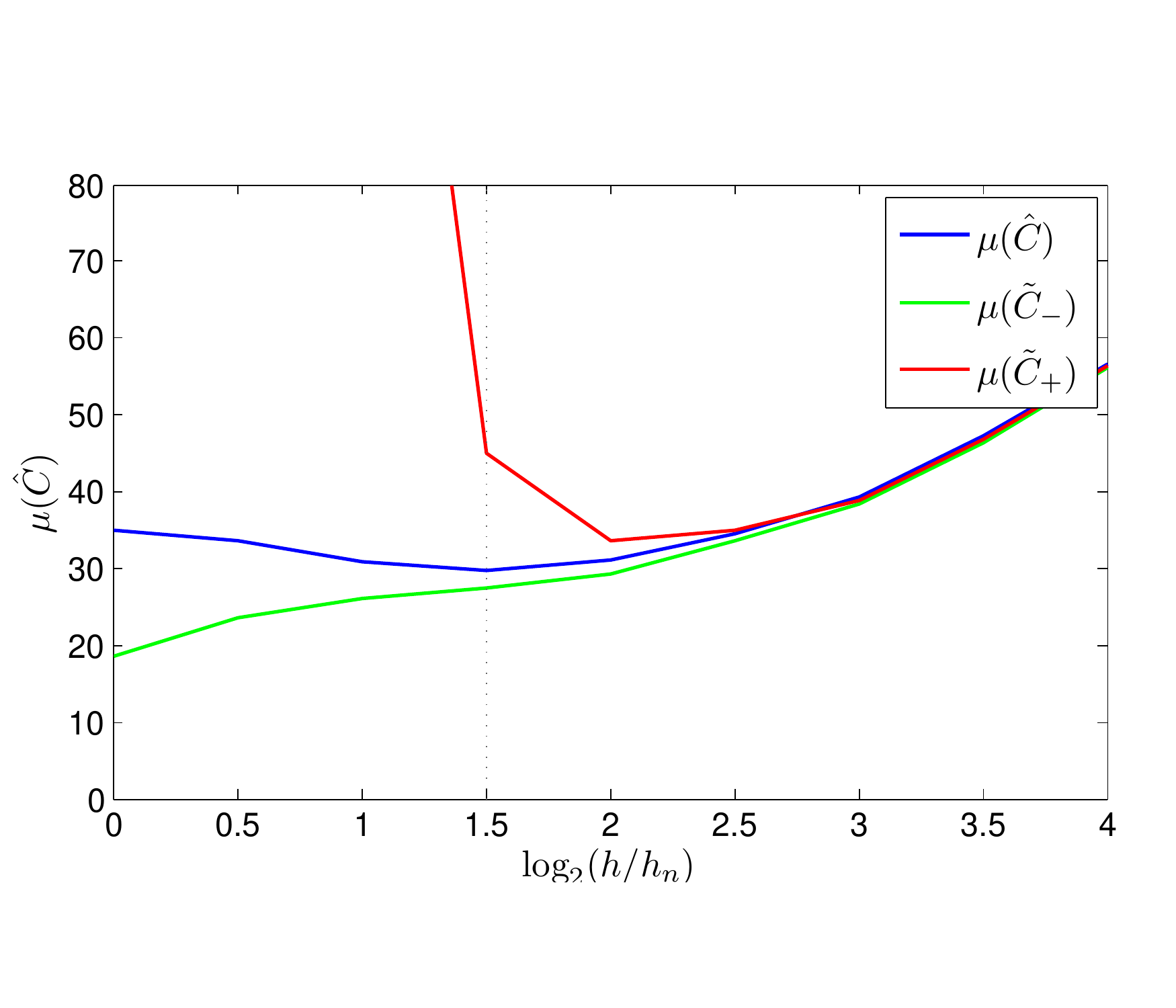}
\includegraphics[scale=0.35]{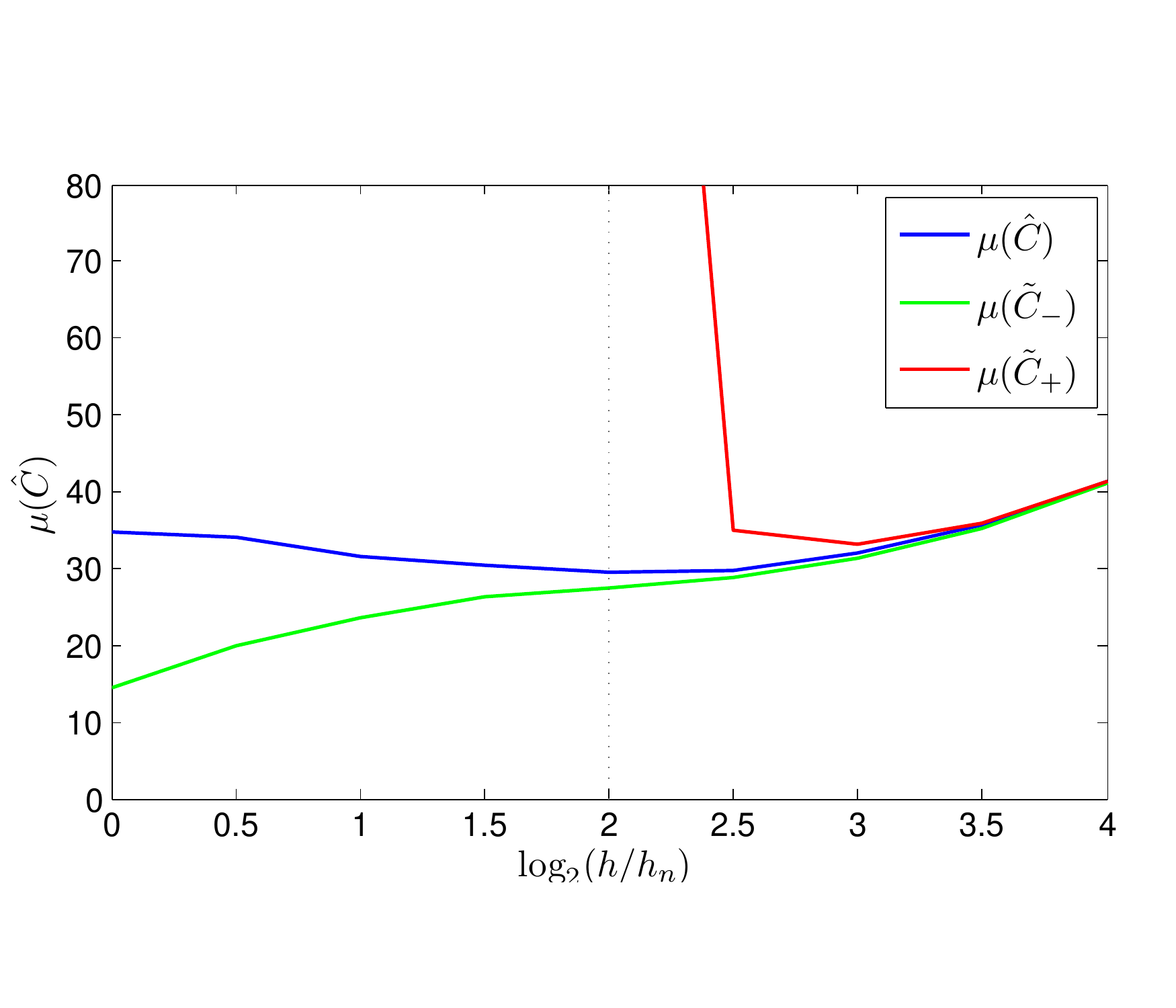}
\end{center}
\caption{Lebesgue measure of the conformal prediction region versus bandwidth
for the Gaussian mixture data with $n=200$ (left) and $n=1000$ (right). Here $h_n=\sqrt{(\log n)/n}$.}
\label{fig:bandwidth}
\end{figure}

\section{Conclusion}\label{sec:conclusion}

We have constructed a distribution free prediction region
by combining ideas from density estimation
and conformal prediction.  It can also be viewed as
a combination of the statistically equivalent block methods
and the density level set methods.
The region is easy to compute
and, under regularity conditions, is asymptotically near optimal.
Even without the regularity conditions,
the region
retains its finite sample validity.

The bandwidth tuning algorithm (Algorithm 3) used in our
simulation resembles cross-validation, a popular device
for kernel density estimators.  In
Algorithm 3, the comparison between candidate
bandwidths is based on a direct evaluation of loss, that is,
the Lebesgue measure of the estimated region.  This feature
yields both conceptually and computationally simple
implementation which is also highly stable as observed in our
simulation
studies.  Future topics of research in this aspect
include understanding
the theoretical properties of such a bandwidth selector,
its connection with other approaches in the literature
of density and level set estimation, and the performance under
both excess loss as well as the
symmetric difference loss.

In current work we are studying nonparametric procedures
that adapt to
smoothness conditions.  In principle it is possible to
further develop this method to deal with nonparametric prediction with
covariates or parametric models.
\newpage

\section{Proofs}\label{sec:proof}
\subsection{Proof of Lemma \ref{lem:sandwich}}\label{subsec:proof_sandwich}
\begin{proof}[Proof of the sandwiching lemma]
The proof is done via a direct characterization of $L_n^-$
and $L_n^+$.

First, for each $y\in L_n^-$ and $i\le i_{n,\alpha}$, we have
\begin{align} &\hat{p}_n^y(y)-\hat{p}_n^y(Y_{(i)})\nonumber\\
=&
\frac{n}{n+1}\left(\hat{p}_n(y)-\hat{p}_n(Y_{(i)})\right)
+\frac{1}{(n+1)h^d}\left(K(0)-K\left(\frac{Y_{(i)}-y}{h}\right)\right)
\nonumber\\
\ge& 0.\nonumber
\end{align}  As a result, $G_n^y(\hat{p}_n^y(y))\ge
i_{n,\alpha}/(n+1)=\tilde\alpha$ and hence
$y\in \hat{C}^{(\alpha)}$.

Similarly, for each $y\notin L_n^+$ and $i \ge i_{n,\alpha}$ we have
\begin{align}
&\hat{p}_h^y(y)-\hat{p}_h^y(Y_{(i)})\nonumber\\=&\frac{n}{n+1}
\left(\hat{p}_h(y)-\hat{p}_h(Y_{(i)})\right) +
\frac{1}{(n+1)h^d}\left(K(0)-K\left(\frac{Y_{(i)}-y}{h}\right)\right)
\nonumber\\
\le& \frac{n}{n+1}\left(\hat{p}_h(y)-\hat{p}_h(Y_{(i_{n,\alpha})})\right)
+\frac{1}{(n+1)h^d}\psi_K \nonumber\\<& 0.\nonumber
\end{align}
Therefore, $G_n^y(\hat{p}_n^y(y))\le (i_{n,\alpha}-1)/(n+1) <\tilde\alpha$
and hence $y\notin \hat{C}^{(\alpha)}$.
\end{proof}

\subsection{Proof of Theorem \ref{thm:rate_of_level}}\label{subsec:proof_rate_of_level}

\paragraph{Preliminaries}
Recall that $L^\ell(t)$ is the lower level set of $p$ at level $t$:
$\{y\in \mathbb{R}^d: p(y)\le t\}$. The bias in the
estimated cut-off level $t_n^{(\alpha)}$ can be
bounded in terms of two quantities:
$$V_n=\sup_{t > 0}|P_n(L^{\ell}(t))-P(L^{\ell}(t))|,$$
and $$R_n=||\hat{p}_n-p||_\infty\,.$$
Here $V_n$ can be viewed as the maximum of the
empirical process $P_n-P$ over a nested class of sets, and
$R_n$ is the $L_\infty$ loss of the density estimator.
As a result,
$V_n$ can be bounded using the standard empirical
process and VC dimension argument, and $R_n$ can be
bounded using the smoothness of $p$ and kernel $K$ with
a suitable choice of bandwidth.
Formally, we provide upper bounds for these two quantities
through the following lemma.

\begin{lemma}\label{lem:V_n}
Let $V_n$, $R_n$ be defined as above,
then under assumptions A1-A3, for any $\lambda>0$,
there exist constants $A_{1,\lambda}$ and $A_{2,\lambda}$ depending
on $\lambda$ only, such that,
$$
\mathbb{P}\left(V_n\ge A_{1,\lambda}\sqrt{\frac{\log n}{n}}\right)=O(n^{-\lambda}),
$$
and
$$
\mathbb{P}\left(R_n\ge A_{2,\lambda}\left(\frac{\log
n}{n}\right)^{\frac{\beta}{2\beta+d}}\right)=O(n^{-\lambda}).
$$
\end{lemma}
\begin{proof}
First, it is easy to check that the class of sets
$\{L^\ell(t):t>0\}$ are nested with VC (Vapnik-Chervonenkis)
dimension 2 and
hence by classical empirical process theory
\citep[see, for example,][Section 2.14]{vdvW96}, there exists a constant
$C_0>0$ such that for
all $\eta >0$
\begin{equation}\label{eq:V_n}
\mathbb{P}(V_n\ge \eta)
\le C_0 n^2\exp(-n\eta^2/32).
\end{equation}
Let $\eta = A\sqrt{\log n / n}$, we have
\begin{align}
  \label{eq:eta_eps}
\mathbb{P}\left(V_n\ge A\sqrt{\log n / n}\right)\le
& C_0n^2\exp(-A^2\log n /32)\nonumber\\
=&C_0 n^{-(A^2/32-2)}\,.
\end{align}
The first result then follows by choosing
 $A_{1,\lambda}=\sqrt{32(\lambda+2)}$.

Next we bound $R_n$.  Let
$\bar{p}=\mathbb{E}[\hat{p}_n]$, and $$\epsilon_n= \left(\frac{\log n}{n}\right)^{\frac{\beta}{2\beta+d}}\,.$$
By triangle inequality
$$
R_n\le ||\hat{p}_n-\bar{p}||_\infty+||\bar{p}-p||_\infty.
$$
Due to a result of \cite{GineG02} \citep[see also equation (49) in Chapter 3 of][]{PrakasaRao83},
under the assumptions
A1(c) and A3, there exist constants $C_1$, $C_2$ and
$B_0>0$ such that have for all $B\ge B_0$,
\begin{align}\label{eq:bound_R_n_var}
\mathbb{P}\left(\|\hat{p}_n-\bar{p}\|_\infty\ge B \epsilon_n\right)
\le & C_1\exp(-C_2B^2\log (h_n^{-1}))\nonumber\\
=&C_1h_n^{C_2B^2}.
\end{align}
On the other hand, by assumptions A1(a) and A3, for some
constant $C_3$
\begin{equation}\label{eq:bound_R_n_bias}
\|\bar{p}-p\|_\infty\le C_3 h_n^\beta.\end{equation}
We note that in the inequalities (\ref{eq:eta_eps}),
(\ref{eq:bound_R_n_var}) and (\ref{eq:bound_R_n_bias})
the constants $C_i$, $i=0,...,3$, depend on $p$ and
$K$ only.
Hence,
\begin{equation}\label{eq:bound_R_n}
\mathbb{P}\left(||\hat{p}-p||_\infty\ge (C_3+B)\epsilon_n\right)
\le C_1h_n^{C_2B^2},\end{equation}
which concludes the second part by choosing
$$A_{2,\lambda}=C_3+\sqrt{\frac{(2\beta+d)\lambda}{C_2}}\,.$$
\end{proof}

\begin{proof}[Proof of Theorem \ref{thm:rate_of_level}]
Let $\alpha_n=i_{n,\alpha}/n=\lfloor (n+1)\alpha\rfloor / n$. We have
$$|\alpha_n - \alpha|\le 1/n\,.$$

Recall that the ideal level $t^{(\alpha)}$ can be written as
$$t^{(\alpha)}=G^{-1}(\alpha)\,,$$ where the function $G$
is the cumulative distribution function of $p(Y)$, as defined in
Subsection \ref{subsec:approx}.
By the modified $\gamma$-exponent
condition the inverse of
$G$ is well defined  in a small neighborhood of $\alpha$.
When $n$ is large enough,
we can
define $t^{(\alpha_n)}$ as
$$t^{(\alpha_n)}=G^{-1}(\alpha_n)\,.$$

Again, by the modified $\gamma$-exponent,
$$c_1|t^{(\alpha_n)}-t^{(\alpha)}|^\gamma\le
|G(t^{(\alpha_n)})-G(t^{(\alpha)})|= |\alpha_n-\alpha|\le n^{-1}.$$
Therefore, for $n$ large enough
\begin{equation}\label{eq:approx_alpha_alpha_n}
|t^{(\alpha_n)}-t^{(\alpha)}|\le (c_1n)^{-1/\gamma}.\end{equation}
Equation (\ref{eq:approx_alpha_alpha_n})
allows us to switch to the problem of bounding
$|t_n^{(\alpha)}-t^{(\alpha_n)}|$.

Recall that $t_n^{(\alpha)}=\hat p_n (Y_{(i_{n,\alpha})})$.
The key of the proof is to observe that
$$t_n^{(\alpha)}=G_n^{-1}(\alpha_n):=\inf\{t:G_n(t)\ge \alpha_n\}\,.$$
Then it suffices to show that $G^{-1}$ and $G_n^{-1}$ are close at
$\alpha_n$.
In fact, by definition of $R_n$ we have for all $t>0$:
$$L^\ell(t-R_n)\subseteq L_n^\ell(t)\subseteq
 L^\ell(t+R_n).$$
Applying the empirical measure $P_n$ to each term in the above:
$$P_n(L^\ell(t-R_n))\le
 P_n(L_n^\ell(t)) \le
  P_n(L^\ell(t+R_n)).$$
By definition of $V_n$,
$$P(L^\ell(t-R_n))-V_n\le P_n(L_n^\ell(t))
 \le P(L^\ell(t+R_n))+V_n.$$
By definition of $G$ and $G_n$, the above inequality
becomes
$$G(t-R_n)-V_n\le G_n(t) \le G(t+R_n)+V_n.$$

Let $W_n=R_n+(2V_n/c_1)^{1/\gamma}$.
Suppose $n$ is large enough such that
$$
\left(\frac{c_1}{n}\right)^{\frac{1}{\gamma}}+
\left(\frac{2A_{1,\lambda}}{c_1}
\sqrt{\frac{\log n}{n}}\right)^{\frac{1}{\gamma}}
<\epsilon_0,
$$
then on the event
$V_n\le A_{1,\lambda} \sqrt\frac{\log n}{n}$,
\begin{align}
  G_n\left(t^{(\alpha_n)}-W_n\right)&\le
  G\left(t^{(\alpha_n)}-W_n+R_n\right)+
  V_n\nonumber\\
  &= G\left(t^{(\alpha_n)}-(2V_n/c_1)^{1/\gamma}\right)
   -G\left(t^{(\alpha_n)}\right)+\alpha_n +V_n \nonumber\\
  & \le \alpha_n-V_n<\alpha_n\,.\nonumber
\end{align}
where the last inequality uses the left side of the modified
$\gamma$-exponent condition.
Similarly,
$G_n(t^{(\alpha_n)}+W_n)> \alpha_n.$
Hence, for $n$ large enough, if
$V_n\le A_{1,\lambda}\sqrt{(\log n)/n}$ then, \begin{equation}\label{eq:bound_t_n_alpha_n}
|t_n^{(\alpha)}-t^{(\alpha_n)}|\le W_n\,.\end{equation}

To conclude the proof, first note that
$$\left(\frac{c_1}{n}\right)^{\frac{1}{\gamma}}=
o\left(\left(\frac{\log n}{n}\right)^{\frac{1}{2\gamma}}\right)\,.$$
Then we can find constant $A_\lambda'$ such that
for all $n$ large enough,
\begin{equation}\label{eq:A_lamda_prime}
\left(A_\lambda'-\left(\frac{2A_{1,\lambda}}{c_1}\right)^{\frac{1}{\gamma}}
\right)\left(\frac{\log
n}{n}\right)^{\frac{1}{2\gamma}}\ge
\left(\frac{c_1}{n}\right)^{\frac{1}{\gamma}}\,.
\end{equation}
Let $A_\lambda=A_{2,\lambda}$. Combining equations (\ref{eq:approx_alpha_alpha_n}) and (\ref{eq:bound_t_n_alpha_n}),
on the event
\begin{equation}\label{eq:E_n_lambda}
E_{n,\lambda}:=\left\{R_n\le A_{\lambda}
\left(\frac{\log n}{n}\right)^{\frac{\beta}{2\beta+d}},
\quad V_n \le A_{1,\lambda}\left(\frac{\log
n}{n}\right)^{\frac{1}{2}}\right\}\,,\end{equation}
we have, for $n$ large enough,
\begin{align}
&|t_n^{(\alpha)}-t^{(\alpha)}|\nonumber\\
\le &|t_n^{(\alpha)}-t^{(\alpha_n)}|+\left(
\frac{c_1}{n}\right)^{\frac{1}{\gamma}}\nonumber\\
\le & W_n+\left(
\frac{c_1}{n}\right)^{\frac{1}{\gamma}}\nonumber\\
\le & R_n +(2c_1^{-1}V_n)^{1/\gamma}+\left(
\frac{c_1}{n}\right)^{\frac{1}{\gamma}}\nonumber\\
\le & A_\lambda \left(\frac{\log n}{n}\right)^{\frac{\beta}{2\beta+d}}
 + \left(\frac{2A_{1,\lambda}}{c_1}
\sqrt{\frac{\log n}{n}}\right)^{\frac{1}{\gamma}}+\left(
\frac{c_1}{n}\right)^{\frac{1}{\gamma}}\nonumber\\
\le &A_\lambda \left(\frac{\log n}{n}\right)^{\frac{\beta}{2\beta+d}}
+ A_\lambda'
\left(\frac{\log n}{n}\right)^{\frac{1}{2\gamma}}\,,\end{align}
where the second last inequality is from the definition of
$E_{n,\lambda}$ and the last inequality is from the choice of
$A_\lambda'$.
The proof is concluded by observing $\mathbb
P(E_{n,\lambda}^c)=O(n^{-\lambda})$, a consequence of
Lemma \ref{lem:V_n}.
\end{proof}

\vspace{0.2 cm}
\subsection{Proof of Theorem \ref{thm:L_n+}}\label{subsec:proof_L_n+}

\begin{proof}[Proof of Theorem \ref{thm:L_n+}] In the proof we write
$t_n$ for $t_n^{(\alpha)}$.
Observe that
\begin{align}
 & \mu\left(L_n(t_n)\triangle C^{(\alpha)}\right)\nonumber\\
  =&\mu\left(\left\{\hat{p}_n\ge t_n,~ p<t^{(\alpha)}\right\}\right) +
  \mu\left(\left\{\hat{p}_n< t_n,~ p\ge t^{(\alpha)}\right\}\right).
\end{align}

Note that
\begin{align}
  \left\{\hat{p}_n\ge t_n,~ p<t^{(\alpha)}\right\}\subseteq \left\{
  t^{(\alpha)}-|t_n-t^{(\alpha)}|-R_n\le p< t^{(\alpha)}\right\},
\end{align}
and
\begin{align}
  \left\{\hat{p}_n< t_n,~ p\ge t^{(\alpha)}\right\}\subseteq
  \left\{t^{(\alpha)}< p\le t^{(\alpha)}+|t^{(\alpha)}-t_n|+R_n\right\}.
\end{align}
Therefore
\begin{align}\label{eq:bound_mu_sym_diff}
  &L_n(t_n)\triangle C^{(\alpha)}\nonumber\\
  \subseteq&\left\{t^{(\alpha)}-|t_n-t^{(\alpha)}|-R_n< p\le t^{(\alpha)}
  +|t^{(\alpha)}-t_n|+R_n\right\}.
\end{align}

Suppose $n$ is large enough such that
$$2A_{2,\lambda}\left(\frac{\log n}{n}\right)^{\frac{\beta}{2\beta+d}}
 +
A_{\lambda}'\left(\frac{\log n}{n}\right)^{\frac{1}{2\gamma}}<
\left(\epsilon_0\wedge \frac{t^{(\alpha)}}{2}\right),$$
where the constant $A_{2,\lambda}$ is defined as in Lemma \ref{lem:V_n} and
$A_\lambda'$ is defined as in equation (\ref{eq:A_lamda_prime}).
Then on the event $E_{n,\lambda}$ as defined in equation
(\ref{eq:E_n_lambda}), applying Theorem
\ref{thm:rate_of_level} and condition (\ref{eq:modif_gamma_exp})
on the right hand side of (\ref{eq:bound_mu_sym_diff}) yields
\begin{align}\mu\left(L_n(t_n)\triangle C^{(\alpha)}
\right)
&\le \frac{P\left(L_n(t_n)\triangle
 C^{(\alpha)}\right)}{t^{(\alpha)}-|t_n-t^{(\alpha)}
 |-R_n}\nonumber\\
&\le\frac{2}{t^{(\alpha)}}c_2\left(2A_{2,\lambda}\left(\frac{\log n}{n}\right)^{\frac{\beta}{2\beta+d}}
 +
A_{\lambda}'\left(\frac{\log n}{n}\right)^{\frac{1}{2\gamma}}\right)^{\gamma}
\nonumber\\
&\le B_\lambda \left(\frac{\log n}{n}\right)^{\frac{\beta\gamma}{2\beta+d}}
+B_\lambda'\left(\frac{\log n}{n}\right)^{\frac{1}{2}},
\end{align}
where $B_\lambda$, $B_\lambda'$ are positive constants depend only on $p$, $K$, $\alpha$
and $\gamma$.
\end{proof}



\bibliographystyle{imsart-nameyear}
\bibliography{conformal}
%
%
%
%
%
%

\end{document}